\documentclass[11pt]{article}

\usepackage{fullpage}
\usepackage{graphicx,amssymb,amsmath}
\usepackage{amsfonts}
\usepackage{latexsym}
\usepackage{amsthm}

\usepackage{datetime}
\usepackage{url}

\usepackage{bbm}
\usepackage[small]{caption}

\newtheorem{theorem}{Theorem}

\newtheorem{proposition}{Proposition}

\def \endproof {\hfill $\Box$ }

\newcommand{\old}[1]{{}}
\newcommand{\later}[1]{{}}

\newcommand{\ie}{{i.e.}}
\newcommand{\eg}{{e.g.}}

\newcommand{\diam}{{\rm diam}}
\newcommand{\per}{{\rm per}}
\newcommand{\esc}{{\rm esc}}
\newcommand{\proj}{{\rm proj}}

\newcommand{\area}{{\rm area}}
\newcommand{\OPT}{{\rm OPT}}

\def\NN{\mathbb{N}}

\def\RR{\mathbb{R}}
\def\SH{\mathbb{S}}

\title{On the Total Perimeter of Homothetic Convex Bodies \\
in a Convex Container\thanks{A preliminary version of this paper
  appeared in \emph{Proceedings of the 17th International Workshop on
    Approximation Algorithms for Combinatorial Optimization Problems}
(APPROX 2013), Berkeley, CA, 2013, LNCS~8096, pp.~96--109.}}

\author{
Adrian Dumitrescu\thanks{Department of Computer Science,
University of Wisconsin--Milwaukee, WI, USA.
Email:~\texttt{dumitres@uwm.edu}.
Supported in part by NSF grant DMS-1001667.}
\qquad
Csaba D. T\'oth\thanks{Department of Mathematics,
California State University, Northridge, Los Angeles, CA, USA; and
Department of Computer Science, Tufts University, Medford, MA, USA.
Email:~\texttt{cdtoth@acm.org}.}
}

\begin{document}
\maketitle

\begin{abstract}
For two planar convex bodies, $C$ and $D$, consider a packing $S$ of
$n$ positive homothets of $C$ contained in $D$. We estimate the total
perimeter of the bodies in $S$, denoted $\per(S)$, in terms of $\per(D)$ and $n$.
When all homothets of $C$ touch the boundary of the container $D$, we show
that either $\per(S)=O(\log n)$ or $\per(S)=O(1)$, depending on how
$C$ and $D$ ``fit together,'' and these bounds are the best possible
apart from the constant factors. Specifically, we establish an optimal bound
$\per(S)=O(\log n)$ unless $D$ is a convex polygon and every side of $D$ is
parallel to a corresponding segment on the boundary of $C$ (for short, $D$
is \emph{parallel to} $C$).
%
When $D$ is parallel to $C$ but the homothets of $C$ may lie anywhere in $D$,
we show that $\per(S)=O((1+\esc(S)) \log n/\log \log n)$, where $\esc(S)$
denotes the total distance of the bodies in $S$ from the boundary of $D$.
Apart from the constant factor, this bound is also the best possible.


\end{abstract}

\section{Introduction} \label{sec:intro}

A finite set $S=\{C_1,\ldots ,C_n\}$ of convex bodies is a
\emph{packing} in a convex body (\emph{container}) $D \subset \RR^2$
if the bodies $C_1,\ldots ,C_n\in S$ are contained in $D$
and they have pairwise disjoint interiors. The term \emph{convex body}
above refers to a compact convex set with nonempty interior in $\RR^2$.
The perimeter of a convex body $C\subset \RR^2$ is denoted $\per(C)$,
and the total perimeter of a packing $S$ is denoted
$\per(S)=\sum_{i=1}^n \per(C_i)$.
Our interest is estimating $\per(S)$ in terms of $n$. In this paper,
we consider packings $S$ that consist of positive homothets of a
convex body $C$. We start with an easy general bound for this case.

\begin{proposition}\label{pro:general}
For every pair of convex bodies, $C$ and $D$,
and every packing $S$ of $n$ positive homothets of $C$ in $D$,
we have $\per(S) \leq \rho(C,D) \sqrt{n}$, where $\rho(C,D)$
depends on $C$ and $D$. Apart from this multiplicative constant,
this bound is the best possible.
\end{proposition}

Our goal is to substantially improve the dependence of $\per(S)$ on $n$
in two different scenarios, motivated by applications to the traveling salesman
problem with neighborhoods (TSPN). In Sections~\ref{sec:disks}--\ref{sec:boundary},
we prove tight bounds on $\per(S)$ in terms of $n$ when all homothets in $S$ touch the
boundary of the container $D$ (see Fig.~\ref{fig:packing}). In Section~\ref{sec:escape},
we prove tight bounds on $\per(S)$ in terms of $n$ \emph{and} the total distance
of the bodies in $S$ from the boundary of $D$. Specifically, for two convex bodies,
$C \subset D \subset \RR^2$, let the \emph{escape distance} $\esc(C)$ be the
distance between $C$ and the boundary of $D$ (Fig.~\ref{fig:squares}, right);
and for a packing $S=\{C_1,\ldots,C_n\}$ in a container $D$,
let $\esc(S)=\sum_{i=1}^n \esc(C_i)$.

\begin{figure}[htbp]
\centering
\includegraphics[width=.75\columnwidth]{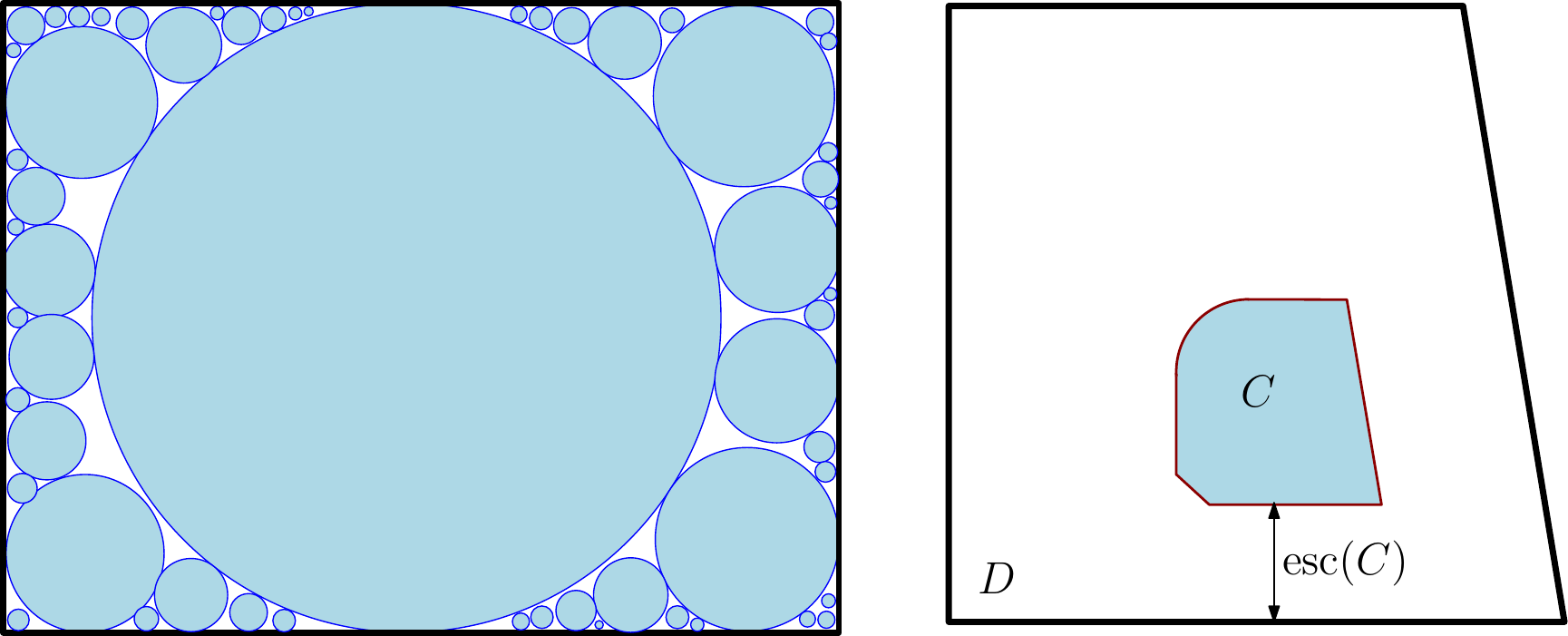}
\caption{Left: a packing of disks in a rectangle container, where all disks
touch the boundary of the container.
Right: a convex body $C$ in the interior of a trapezoid $D$ at
distance $\esc(C)$ from the boundary of $D$. The trapezoid $D$ is
\emph{parallel to} $C$: every side of $D$ is parallel and
``corresponds'' to a side of $C$.}
\label{fig:packing}
\end{figure}

\paragraph{Homothets touching the boundary of a convex container.}
We would like to bound $\per(S)$ in terms of $\per(D)$ and $n$
when all homothets in $S$ touch the boundary of $D$ (see Fig.~\ref{fig:packing}).
Specifically, for a pair of convex bodies, $C$ and $D$,
let $f_{C,D}(n)$ denote the maximum perimeter $\per(S)$ of a packing
of $n$ positive homothet of $C$ in the container $D$, where each element
of $S$ touches the boundary of $D$. We would like to estimate the growth
rate of $f_{C,D}(n)$ as $n$ goes to infinity. We prove a
logarithmic\footnote{Throughout this paper, $\log x$ denotes the logarithm of $x$ to base 2.}
 upper bound $f_{C,D}(n)=O(\log n)$ for every pair of convex bodies, $C$ and $D$.

\begin{proposition}\label{pro:boundary}
For every pair of convex bodies, $C$ and $D$,
and every packing $S$ of $n$ positive homothets of $C$ in $D$,
where each element of $S$ touches the boundary of $D$, we have
$\per(S) \leq \rho(C,D) \log n$, where $\rho(C,D)$ depends on $C$ and~$D$.
\end{proposition}

The upper bound $f_{C,D}(n)=O(\log n)$ is asymptotically tight for
some pairs $C$ and $D$, and not so tight for others. For example, it is not
hard to attain an $\Omega(\log n)$ lower bound when $C$ is an axis-aligned
square, and $D$ is a triangle (Fig.~\ref{fig:squares}, left). However,
$f_{C,D}(n)=\Theta(1)$ when both $C$ and $D$ are axis-aligned squares.
We start by establishing a logarithmic lower bound in the
simple setting where $C$ is a circular disk and $D$ is a unit square.
\begin{theorem}\label{thm:disks}
The total perimeter of $n$ pairwise disjoint disks lying in
the unit square $U=[0,1]^2$ and touching the boundary of $U$ is
$O(\log{n})$. Apart from the constant factor, this bound is the best
possible.
\end{theorem}

We determine $f_{C,D}(n)$ up to constant factors for all pairs of
convex bodies of bounded description complexity\footnote{A planar set
has \emph{bounded description complexity} if its boundary consists of
a finite number of algebraic curves of bounded degrees.}.
We show that either $f_{C,D}=\Theta(\log n)$ or $f_{C,D}(n)=\Theta(1)$
depending on how $C$ and $D$ ``fit together''. To distinguish these
cases, we need the following definitions.

\begin{figure}[htbp]
\centering
\includegraphics[width=.55\columnwidth]{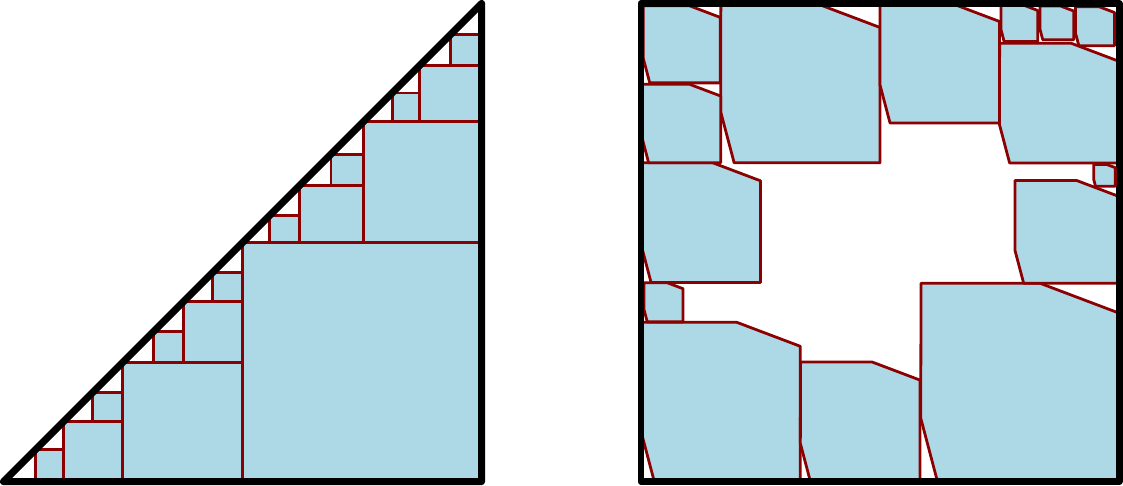}
\caption{Left: a square packing in a triangle where every square
  touches the boundary of the triangle.
Right: a packing of homothetic hexagons $H$ in a square $U$, where $U$
is parallel to $H$ and every hexagon touches the boundary of $U$.}
\label{fig:squares}
\end{figure}

\paragraph{Definition of ``parallel'' convex bodies.}
For a direction vector $\mathbf{d}\in \SH$ and a convex body $C$,
the \emph{supporting line} $\ell_{\mathbf{d}}(C)$ is a directed line of
direction $\mathbf{d}$ such that $\ell_{\mathbf{d}}(C)$ is tangent to
$C$, and the closed halfplane on the left of $\ell_{\mathbf{d}}(C)$
contains $C$. If $\ell_{\mathbf{d}}(C) \cap C$ is a nondegenerate
line segment, we refer to it as a \emph{side} of $C$.

We say that a convex polygon (container) $D$ is \emph{parallel to} a
convex body $C$ when for every direction $\mathbf{d}\in \SH$ if
$\ell_\mathbf{d}(D)\cap D$ is a side of $D$,
then $\ell_\mathbf{d}(C)\cap C$ is also a side of $C$.
Figure~\ref{fig:squares}~(right) depicts a square $D$ parallel to a
convex hexagon $C$.
For example, every positive homothet of a convex polygon $P$ is parallel to $P$;
and all axis-aligned rectangles are parallel to each other.

\paragraph{Classification.}
We generalize the lower bound construction in Theorem~\ref{thm:disks}
to arbitrary convex bodies, $C$ and $D$, of bounded description complexity,
where $D$ is not parallel to $C$.

\begin{theorem}\label{thm:boundary}
Let $C$ and $D$ be two convex bodies of bounded description
complexity. For every packing $S$ of $n$ positive homothets of $C$ in $D$,
where each element of $S$ touches the boundary of $D$, we have
$\per(S) \leq \rho(C,D) \log n$, where $\rho(C,D)$ depends on $C$ and $D$.
Apart from the factor $\rho(C,D)$, this bound is the best possible
unless $D$ is a convex polygon parallel to $C$.
\end{theorem}

If $D$ is a convex polygon parallel to $C$, and every homothet of $C$
in a packing $S$ of $n$ homothets touches the boundary of $D$, then it
is not difficult to see that $\per(S)$ is bounded from above by an expression
independent of $n$.

\begin{proposition}\label{pro:parallel}
Let $C$ and $D$ be convex bodies such that $D$ is a convex polygon
parallel to $C$. Then every packing $S$ of $n$ positive homothets of
$C$ in $D$, where each element of $S$ touches the boundary of $D$, we
have $\per(S) \leq \rho(C,D)$, where $\rho(C,D)$ depends on $C$ and~$D$.
\end{proposition}

\paragraph{Total distance form the boundary of a convex container.}
In the general case, when the homothets of $C$ can be in the interior
of the container $D$, we improve the dependence on $n$ of the general
bound Proposition~\ref{pro:general} by using the escape distance,
namely the total distance of the homothets of $C$ from the boundary of $D$.
Combining the bound in Proposition~\ref{pro:general}
with inequality~\eqref{pro:boundary} yields the following bound.

\begin{proposition}\label{pro:homothets}
For every pair of convex bodies, $C$ and $D$,
and every packing $S$ of $n$ positive homothets of $C$ in $D$,
we have $\per(S) \leq \rho(C,D) (\esc(S)+\log n)$, where $\rho(C,D)$
depends on $C$ and~$D$.
\end{proposition}

By Theorem~\ref{thm:boundary}, the logarithmic upper bound in terms of $n$ is
the best possible when $D$ is not parallel to $C$. When $D$ is a convex
polygon parallel to $C$, we derive the following upper bound for $\per(S)$,
which is also asymptotically tight in terms of $n$.

\begin{theorem}\label{thm:homothets}
Let $C$ and $D$ be two convex bodies such that $D$ is a
convex polygon parallel to $C$. For every packing $S$ of $n$ positive
homothets of $C$ in $D$, we have
$$\per(S) \leq \rho(C,D) \left(\per(D) +\esc(S)\right) \, \frac{\log n}{\log \log n},$$
where $\rho(C,D)$ depends on $C$ and $D$.
For every $n\geq 1$, there exists a packing $S$ of $n$ positive homothets of
$C$ in $D$ such that
$\per(S) \geq \sigma(C,D) \left(\per(D) +\esc(S)\right) \, \frac{\log n}{\log \log n},$
where $\sigma(C,D)$ depends on $C$ and $D$.
\end{theorem}

\paragraph{Related Previous Work.}
We consider the total perimeter $\per(S)$ of a packing $S$ of $n$ homothets of
a convex body $C$ in a convex container $D$ in Euclidean plane.
Other variants have also been considered:
(1) If $S$ is a packing of $n$ \emph{arbitrary} convex bodies in $D$,
then it is easy to subdivide $D$ by $n-1$ near diameter segments
into $n$ convex bodies of total perimeter close to
$\per(D)+2(n-2)\diam(D)$. Glazaryn and Mori\'c~\cite{GM14} have recently
proved that this lower bound is the best possible when $D$ is a square or a triangle.
For an arbitrary convex body $D$, they prove an upper bound of
$\per(S) \leq 1.22\, \per(D) + 2(n-1) \diam(D)$.
(2) If all bodies in $S$ are congruent to a convex body $C$, then
$\per(S)=n \, \per(C)$, and bounding $\per(S)$ from above reduces to the
classical problem of determining the maximum number of
interior-disjoint congruent copies of $C$ that fit in $D$~\cite[Section~1.6]{BMP05}.

Considerations of the total surface area of a ball packing in $\RR^3$
also play an important role in a strong version of the Kepler
conjecture~\cite{Bez13,Hal13}.

\paragraph{Motivation.}
In the \emph{Euclidean Traveling Salesman Problem} (ETSP),
given a set $S$ of $n$ points in $\RR^d$, one wants to find
a closed polygonal chain (\emph{tour}) of minimum Euclidean length
whose vertex set is $S$. The Euclidean TSP is known to be NP-hard,
but it admits a PTAS in $\RR^d$, where $d\in \NN$ is constant~\cite{Ar98}.
In the \emph{TSP with Neighborhoods}
(TSPN), given a set of $n$ sets (neighborhoods) in $\RR^d$,
one wants to find a closed polygonal chain of minimum Euclidean
length that has a vertex in each neighborhood. The neighborhoods
are typically simple geometric objects (of bounded description complexity)
such as disks, rectangles, line segments, or lines. TSPN is know to be NP-hard;
and it admits a PTAS for certain types of neighborhoods~\cite{Mi10},
but is hard to approximate for others~\cite{BGK+05}.

For $n$ connected (possibly overlapping) neighborhoods in the plane,
TSPN can be approximated with ratio $O(\log n)$ by an algorithm of
Mata and Mitchell~\cite{MM95}. See also the survey by Bern and Eppstein~\cite{BE97}
for a short outline of this algorithm. At its core, the $O(\log{n})$-approximation
relies on the following early result by Levcopoulos and Lingas~\cite{LL84}:
every (simple) rectilinear polygon $P$ with $n$ vertices, $r$ of which are reflex,
can be partitioned into rectangles of total perimeter $O(\per(P) \log r)$
in $O(n \log n)$ time.

A natural approach for finding a solution to TSPN is the
following~\cite{DM03,DT13} (in particular, it achieves a
constant-ratio approximation for unit disks):
Given a set $S$ of $n$ neighborhoods, compute a maximal subset
$I \subseteq S$ of pairwise disjoint neighborhoods
(\ie, a packing), compute a good tour for $I$, and then augment it by
traversing the boundary of each set in $I$. Since each neighborhood in
$S \setminus I$ intersects some neighborhood in $I$, the augmented
tour visits all members of $S$. This approach is particularly
appealing since good approximation algorithms are often available for
pairwise disjoint neighborhoods~\cite{Mi10}. The bottleneck of this
approach is the length increase incurred by extending a tour of $I$ by
the total perimeter of the neighborhoods in $I$. An upper bound
$\per(I)=o(\OPT(I)\log n)$ would immediately imply an improved
$o(\log n)$-factor approximation ratio for TSPN.

Theorem~\ref{thm:boundary} shows that this approach cannot beat the
$O(\log n)$ approximation ratio for most types of neighborhoods
(\eg, circular disks). In the current formulation,
Proposition~\ref{pro:boundary} yields the upper bound $\per(I)=O(\log n)$
assuming a convex container, so in order to use this bound, a tour of
$I$ needs to be augmented into a convex partition; this may increase
the length by a $\Theta(\log n/\log \log n)$-factor in the worst
case~\cite{DT11,LL84}. For convex polygonal neighborhoods, the bound
$\per(I)=O(1)$ in Proposition~\ref{pro:parallel} is applicable after a
tour for $I$ has been augmented into a convex partition with
\emph{parallel} edges (\eg, this is possible for axis-aligned
rectangle neighborhoods, and an axis-aligned approximation of the
optimal tour for $I$). The convex partition of a polygon with $O(1)$
distinct orientations, however, may increase the length by a
$\Theta(\log n)$-factor in the worst case~\cite{LL84}.
Overall our results show that we cannot beat the current
$O(\log n)$ ratio for TSPN for any type of homothetic neighborhoods
if we start with an arbitrary independent set $I$ and an arbitrary
near-optimal tour for $I$.

\section{Preliminaries: A Few Easy Pieces}\label{sec:prelim}

\noindent{\bf Proof of Proposition~\ref{pro:general}.}
Let $\mu_i>0$ denote the homothety factor of $C_i$, \ie, $C_i =\mu_i C$,
for $i=1,\ldots,n$. Since $S$ is a packing we have
$\sum_{i=1}^n \mu_i^2 \area(C) \leq \area(D)$.
By the Cauchy-Schwarz inequality we have
$(\sum_{i=1}^n \mu_i)^2 \leq n \sum_{i=1}^n \mu_i^2$.
It follows that
\begin{align*}
\per(S) &= \sum_{i=1}^n \per(C_i) =
\per(C) \sum_{i=1}^n \mu_i \\
&\leq \per(C) \sqrt{n} \sqrt{\left(\sum_{i=1}^n \mu_i^2\right)}
\leq \per(C) \sqrt{\frac{\area(D)}{\area(C)}} \, \sqrt{n}.
\end{align*}
Set now $\rho(C,D):=\per(C) \sqrt{\area(D)/\area(C)}$, and the
proof of the upper bound is complete.

For the lower bound, consider two convex bodies, $C$ and $D$.
Let $U$ be a maximal axis-aligned square inscribed in $D$, and let
$\mu C$ be the largest positive homothet of $C$ that fits into $U$.
Note that $\mu=\mu(C,D)$ is a constant that depends on $C$ and $D$ only.
Subdivide $U$ into $\lceil \sqrt{n}\rceil^2$ congruent copies of the
square $\frac{1}{\lceil \sqrt{n}\rceil}U$. Let $S$ be
the packing of $n$ copies of $\frac{\mu}{\lceil \sqrt{n}\rceil}C$
(\ie, $n$ translates),
with at most one in each square $\frac{1}{\lceil \sqrt{n}\rceil}U$.
The total perimeter of the packing is
$\per(S)=n\cdot \frac{\mu}{\lceil \sqrt{n}\rceil}\per(C)=\Theta(\sqrt{n})$,
as claimed.
\endproof

\paragraph{Proof of Proposition~\ref{pro:boundary}.}
Let $S=\{C_1,\ldots , C_n\}$ be a packing of $n$ homothets of $C$ in $D$
where each element of $S$ touches the boundary of $D$.
Observe that $\per(C_i)\leq \per(D)$ for all $i=1,\ldots  ,n$.
Partition the elements of $S$ into subsets as follows.
For $k=1,\ldots , \lceil \log n\rceil$, let $S_k$ denote the set of
homothets $C_i$ such that $\per(D)/2^k<\per(C_i)\leq \per(D)/2^{k-1}$;
and let $S_0$ be the set of homothets $C_i$ of perimeter less than
$\per(D)/2^{\lceil \log n\rceil}$.
Then the sum of perimeters of the elements in $S_0$ is
$\per(S_0)\leq n \, \per(D)/2^{\lceil \log n\rceil} \leq \per(D)$,
since $S_0 \subseteq S$ contains at most $n$ elements altogether.

For $k=1,\ldots , \lceil \log n\rceil$, the diameter of each
$C_i\in S_k$ is bounded above by
\begin{equation} \label{E1}
\diam(C_i)<\per(C_i)/2\leq \per(D)/2^k.
\end{equation}
Consequently, every point of a body $C_i\in S_k$ lies at
distance at most $\per(D)/2^k$ from the boundary of $D$,
denoted $\partial D$. Let $R_k$ be the set of points in $D$ at
distance at most $\per(D)/2^k$ from $\partial D$. Then
\begin{equation} \label{E2}
\area(R_k)\leq \per(D) \, \frac{\per(D)}{2^k} = \frac{(\per(D))^2}{2^k}.
\end{equation}
Since $S$ consists of homothets, the area of any element $C_i \in S_k$ is
bounded from below by
\begin{equation} \label{E3}
\area(C_i) = \area(C) \left(\frac{\per(C_i)}{\per(C)} \right)^2
\geq \area(C) \left(\frac{\per(D)}{2^k \, \per(C)} \right)^2.
\end{equation}
By a volume argument, \eqref{E2} and~\eqref{E3} yield
\begin{equation*}
|S_k|\leq
\frac{\area(R_k)}{\min_{C_i\in S_k} \area(C_i)} \leq
\frac{(\per(D))^2/2^k}{\area(C)(\per(D))^2/(2^k \, \per(C))^2}=
\frac{(\per(C))^2}{\area(C)} \cdot 2^k.
\end{equation*}
Since for $C_i \in S_k$, $k=1,\ldots,\lceil \log n\rceil$,
we have $\per(C_i)\leq \per(D)/2^{k-1}$,  it follows that
$$ \per(S_k)\leq |S_k| \cdot \frac{\per(D)}{2^{k-1}} \leq
2 \, \frac{(\per(C))^2}{\area(C)} \, \per(D). $$

Hence the sum of perimeters of all elements in $S$ is bounded by
$$\per(S)=\sum_{k=0}^{\lceil \log n\rceil} \per(S_k) \leq
 \left(1+2 \, \frac{(\per(C))^2}{\area(C)} \lceil \log n\rceil
\right) \per(D), $$
as required.
\endproof

\paragraph{Proof of Proposition~\ref{pro:parallel}.}
Let $\rho'(C)$ denote the ratio between $\per(C)$ and the length of a
shortest side of $C$. Recall that each $C_i\in S$ touches the boundary
of polygon $D$. Since $D$ is parallel to $C$, the side of $D$ that supports
$C_i$ must contain a side of $C_i$. Let $a_i$ denote the length of this side.
$$ \per(S)=\sum_{i=1}^n \per(C_i) = \sum_{i=1}^n a_i \, \frac{\per(C_i)}{a_i}
\leq \rho'(C) \sum_{i=1}^n a_i \leq \rho'(C) \, \per(D). $$
Set now $\rho(C,D):=\rho'(C) \, \per(D)$, and the proof is complete.
\endproof

\paragraph{Proof of Proposition~\ref{pro:homothets}.}
The proof is similar to that of Proposition~\ref{pro:boundary} with
a few adjustments. Let $S=\{C_1,\ldots , C_n\}$ be a packing of $n$
homothets of $C$ in $D$. Note that $\per(C_i)\leq \per(D)$ for all $i=1,\ldots  ,n$.
Partition the elements of $S$ into subsets as follows. Let
$$S^{\rm in}=\{C_i\in S: \per(C_i)\leq \esc(C_i)\} \textup{ and }
S^{\rm bd} = S\setminus S^{\rm in}.$$
For $k=1,\ldots , \lceil \log n\rceil$, let $S_k$ denote the set of
homothets $C_i\in S^{\rm bd}$ such that $\per(D)/2^k<\per(C_i)\leq \per(D)/2^{k-1}$;
and let $S_0$ be the set of homothets $C_i\in S^{\rm bd}$ of perimeter at most
$\per(D)/2^{\lceil \log n\rceil}$.

The sum of perimeters of the elements in $S^{\rm in}$ is
$\per(S^{\rm in}) \leq \esc(S^{\rm in})\leq \esc(S)$.
We next consider the elements in $S^{\rm bd}$.
The sum of perimeters of the elements in $S_0$ is
$\per(S_0)\leq n \, \per(D)/2^{\lceil \log n\rceil} \leq \per(D)$,
since $S_0 \subseteq S$ contains at most $n$ elements altogether.

For $k=1,\ldots , \lceil \log n\rceil$, the diameter of each
$C_i\in S_k$ is bounded above by $\diam(C_i)<\per(C_i)/2 \leq \per(D)/2^k$.
Observe that every point of a body $C_i\in S_k$ lies at
distance at most
$\esc(C_i) + \diam(C_i) \leq \per(C_i) + \diam(C_i)
\leq 1.5 \, \per(C_i) \leq 3 \, \per(D)/2^k$ from $\partial D$.
Let now $R_k$ be the set of points in $D$ at
distance at most $3\, \per(D)/2^k$ from $\partial D$.
Then
$$ \area(R_k)\leq \per(D) \, \frac{3\per(D)}{2^k} =
\frac{3\, (\per(D))^2}{2^k}. $$
Analogously to the proof of Proposition~\ref{pro:boundary},
a volume argument yields
$$ |S_k| \leq 3\, \frac{(\per(C))^2}{\area(C)} \cdot 2^k. $$

It follows that
$$ \per(S_k)\leq |S_k| \cdot \frac{\per(D)}{2^{k-1}} \leq
6 \, \frac{(\per(C))^2}{\area(C)} \, \per(D). $$
Hence the sum of perimeters of all elements in $S$ is bounded by
$$ \per(S)\leq \esc(S)
+\left(1+6 \, \frac{(\per(C))^2}{\area(C)} \lceil \log n\rceil \right) \per(D), $$
as required.
\endproof

\section{Disks Touching the Boundary of a Square:
Proof of Theorem~\ref{thm:disks}}\label{sec:disks}

Let $S$ be a set of $n$ interior-disjoint disks
in the unit square $U=[0,1]^2$ that touch the boundary of $U$.
From Proposition~\ref{pro:boundary} we deduce the upper
bound $\per(S)=O(\log n)$, as required.
To prove the matching lower bound, it remains to construct a packing of $O(n)$ disks
in the unit square $U$ such that every disk touches the $x$-axis, and the sum of their
diameters is $\Omega(\log{n})$. We present two constructions attaining this bound:
(i) an \emph{explicit} construction in Subsection~\ref{ssec:explicit} which
will be generalized in Section~\ref{sec:boundary};
and (ii) a \emph{greedy} disk packing.

\subsection{An Explicit Construction}\label{ssec:explicit}

For convenience, we use the unit square $[-\frac{1}{2},\frac{1}{2}]\times[0,1]$
for our construction. To each disk we associate its vertical \emph{projection
interval} (on the $x$-axis). The algorithm greedily chooses disks of
monotonically decreasing radii such that (1) every diameter is $1/16^k$ for some
$k \in \NN$; and  (2) if the projection intervals of two disks overlap,
then one interval contains the other.

For $k=0,1,\ldots , \lfloor \log_{16} n\rfloor$,
denote by $S_k$ the set of disks of diameter $1/16^k$,
constructed by our algorithm. We recursively allocate
a set of intervals $X_k\subset [-\frac{1}{2},\frac{1}{2}]$ to $S_k$,
and then choose disks in $S_k$ such that their projection
intervals lie in $X_k$. Initially, $X_0=[-\frac{1}{2},\frac{1}{2}]$,
and $S_0$ contains the disk of diameter 1
inscribed in $[-\frac{1}{2},\frac{1}{2}]\times[0,1]$.
The length of each maximal interval $I\subseteq X_k$
will be a multiple of $1/16^k$, so $I$ can be
covered by projection intervals of interior-disjoint disks
of diameter $1/16^k$ touching the $x$-axis. Every interval
$I \subseteq X_k$ will have the property that any disk of diameter
$1/16^k$ whose projection interval is in $I$ is disjoint from any
(larger) disk in $S_j$, $j<k$.
\begin{figure}[htbp]
\centering
\includegraphics[width=.6\columnwidth]{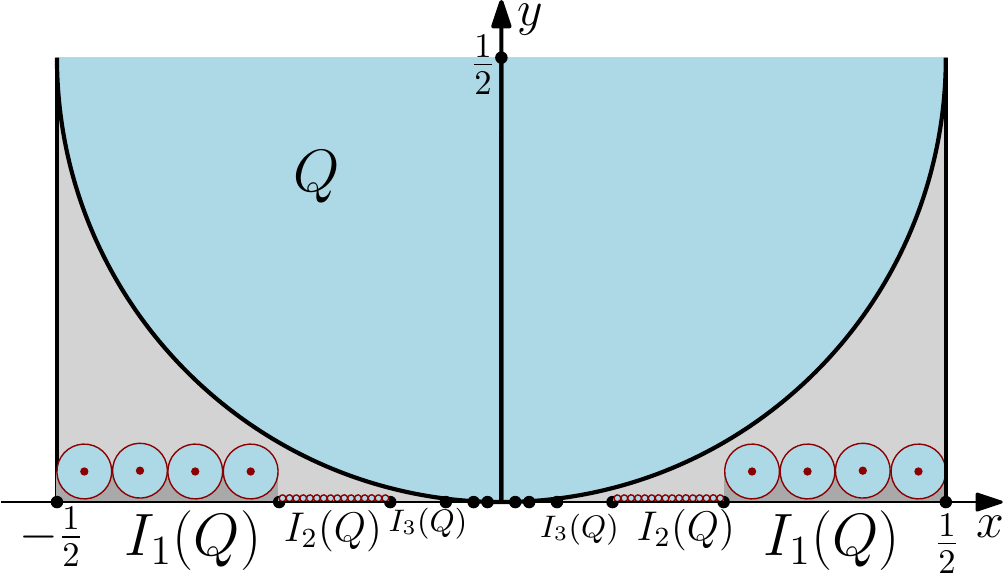}
\caption{Disk $Q$ and the exponentially decreasing pairs of
  intervals $I_k(Q)$, $k=1,2,\ldots$.}
\label{fig:disks}
\end{figure}

Consider the disk $Q$ of diameter 1, centered at $(0,\frac{1}{2})$, and tangent
to the $x$-axis (see Fig.~\ref{fig:disks}).
It can be easily verified that:
\begin{itemize}
\itemsep -2pt
\item[{\rm (i)}] the locus of centers of disks tangent to both $Q$ and the
$x$-axis is the parabola $y=\frac{1}{2}x^2$; and
\item[{\rm (ii)}] any disk of diameter $1/16$ and tangent to the
  $x$-axis whose projection interval is in
  $I_1(Q)=[-\frac{1}{2},-\frac{1}{4}] \cup [\frac{1}{4},\frac{1}{2}]$
is disjoint from $Q$.
\end{itemize}
Indeed, the center of any such disk is $(x_1,\frac{1}{16})$, for
$x_1\leq -\frac{5}{16}$ or $x_1\geq\frac{5}{16}$, and hence lies below
the parabola $y=\frac{1}{2}x^2$.
Similarly, for all $k\in \NN$, any disk of diameter $1/16^k$
and tangent to the $x$-axis whose projection interval is in
$I_k(Q)=[-\frac{1}{2^{k}},-\frac{1}{2^{k+1}}]
\cup [\frac{1}{2^{k+1}},\frac{1}{2^{k}}]$ is disjoint from $Q$.
For an arbitrary disk $D$ tangent to the $x$-axis, and an integer $k \geq
1$, denote by $I_k(D) \subseteq [-\frac{1}{2},\frac{1}{2}]$ the pair of
intervals corresponding to $I_k(Q)$; for $k=0$, $I_k(D)$ consists of only
one interval.

We can now recursively allocate intervals in $X_k$ and choose disks in $S_k$
($k=0,1,\ldots , \lfloor \log_{16} n\rfloor$) as follows. Recall that
$X_0=[-\frac{1}{2},\frac{1}{2}]$, and $S_0$ contains a single disk of unit
diameter inscribed in the unit square $[-\frac{1}{2},\frac{1}{2}]\times[0,1]$.
Assume that we have already defined the intervals in $X_{k-1}$, and
selected disks in $S_{k-1}$. Let $X_k$ be the union of the interval
pairs $I_{k-j}(D)$ for all $D \in S_j$ and $j=0,1,\ldots , k-1$.
Place the maximum number of disks of diameter $1/16^k$ into $S_k$ such that
their projection intervals are contained in $X_k$. For a disk $D\in S_j$
($j=0,1,\ldots , k-1$) of diameter $1/16^j$, the two intervals in
$X_{k-j}$ each have length $\frac{1}{2}\cdot \frac{1}{2^{k-j}}\cdot \frac{1}{16^j}
=\frac{8^{k-j}}{2}\cdot \frac{1}{16^k}$, so they can each accommodate
the projection intervals of $\frac{8^{k-j}}{2}$ disks in $S_k$.

We prove by induction on $k$ that the length of $X_k$ is $\frac{1}{2}$,
and so the sum of the  diameters of the disks in $S_k$ is $\frac{1}{2}$,
$k=1,2,\ldots, \lfloor \log_{16}{n}\rfloor$. The interval
$X_0=[-\frac{1}{2},\frac{1}{2}]$ has length 1. The pair of intervals
$X_1=[-\frac{1}{2},-\frac{1}{4}] \cup [\frac{1}{4},\frac{1}{2}]$ has
length $\frac{1}{2}$. For $k=2,\ldots, \lfloor \log_{16}n\rfloor$,
the set $X_k$ consists of two types of (disjoint) intervals:
(a) The pair of intervals $I_1(D)$ for every $D\in S_{k-1}$ covers half of
the projection interval of $D$. Over all $D\in S_{k-1}$, they jointly
cover half the length of $X_{k-1}$. (b) Each pair of intervals
$I_{k-j}(D)$ for $D\in S_{k-j}$, $j=0,\ldots ,k-2$,
has half the length of $I_{k-j-1}(D)$. So the sum of the lengths of
these intervals is half the length of $X_{k-1}$; although they are
disjoint from $X_{k-1}$.
Altogether, the sum of lengths of all intervals in $X_k$ is the same
as the length of $X_{k-1}$. By induction, the length of $X_{k-1}$
is $\frac{1}{2}$, hence the length of $X_k$ is also $\frac{1}{2}$,
as claimed. This immediately implies that the sum of diameters
of the disks in $\bigcup_{k=0}^{\lfloor \log_{16} n\rfloor} S_k$
is $1+\frac{1}{2}\lfloor \log_{16} n\rfloor$.
Finally, one can verify that the total number of disks used is $O(n)$.
Write $K=\lfloor \log_{16} n\rfloor$.
Indeed, $|S_0|=1$, and $|S_k|=|X_k|/16^{-k}=16^k/2$, for
$k=1,\ldots,K$, where $|X_k|$ denotes the
total length of the intervals in $X_k$. Consequently,
$|S_0|+ \sum_{k=1}^{K}  |S_k|= O(16^k)= O(n)$, as required.
\endproof

\subsection{A Greedy Disk Packing}\label{ssec:greedy}

The following simple greedy algorithm produces a packing $S_n$ of $n$ disks in
the unit square $U=[0,1]^2$ with all disks touching the boundary of $U$ and
whose total perimeter is $\Omega(\log{n})$. For $i=1$ to $n$, let
$C_i$ be a disk of maximum radius that lies in $U\setminus (\bigcup_{j<i}C_j)$
and intersects $\partial U$, and let $S_n=\{C_1,\ldots, C_n\}$;
refer to Fig.~\ref{fig:greedy}~(left). The radius of $C_1$ is $1/2$,
the radii of $C_2,\ldots, C_5$ are $3-2\sqrt{2}$, etc.
We use Apollonian circle packings~\cite{GLM+05} to derive the lower bound
$\per(S_n) =\Omega(\log{n})$.

We now consider a greedy algorithm in a slightly different setting,
applicable to Apollonian circles.
For $r_1,r_2>0$, we construct a set $F_n(r_1,r_2)$ of $n$ disks by the
following greedy algorithm. Let $A_1$ and $A_2$ be two tangent disks
of radii $r_1$ and $r_2$ that are also tangent to the $x$-axis from above.
Let $I$ be the horizontal segment between the tangency points of
$A_1$ and $A_2$ with the $x$-axis. For $i=3,\ldots,n$, let $A_i$ be the
disk of maximum radius tangent to segment~$I$,
lying above the $x$-axis, and disjoint from the interior of all
disks $A_j$, $j<i$. See Fig.~\ref{fig:greedy}~(right), where $r_1=r_2=1/2$.
We now compare the total perimeter of the two greedy disk packings
described above.
\begin{figure}[htbp]
\centering
\includegraphics[width=.8\textwidth]{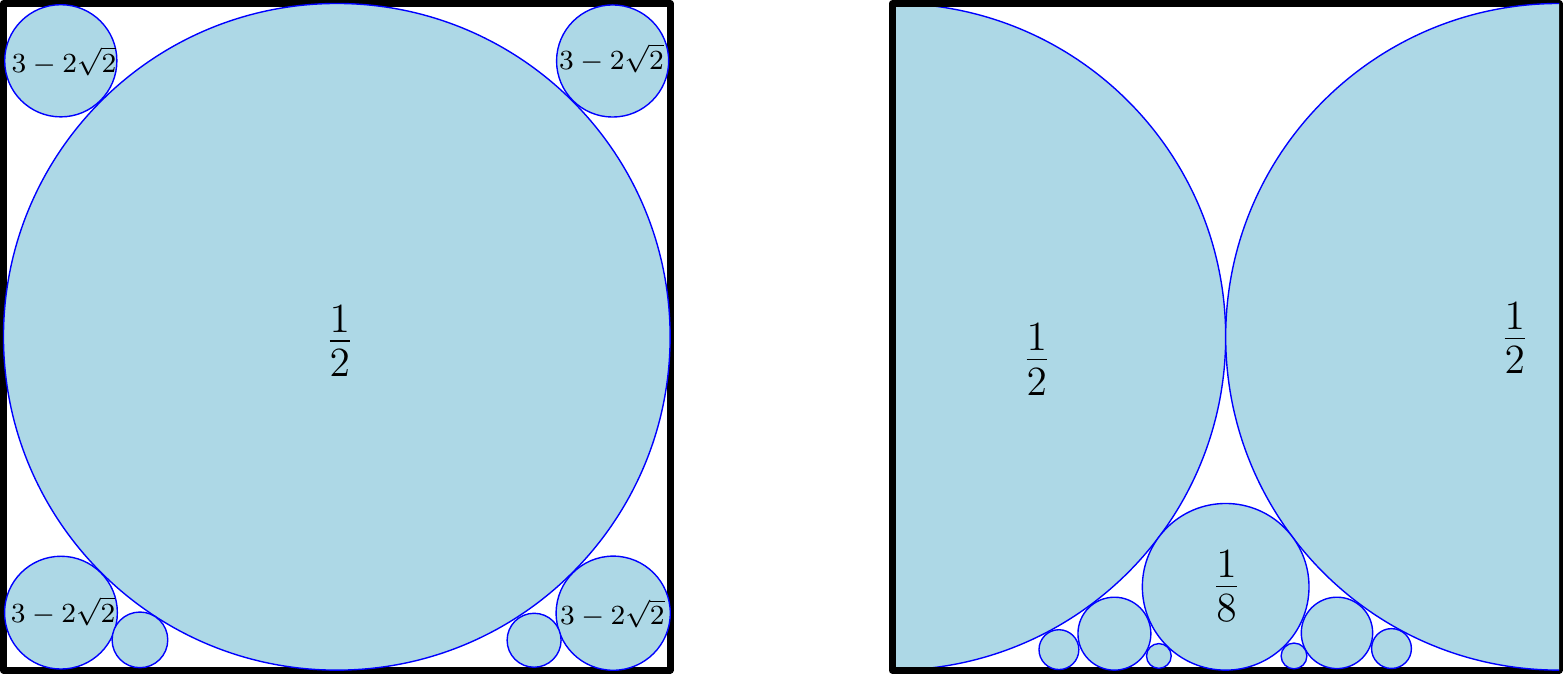}
\caption{Left: A greedy packing of $n=7$ disks in $[0,1]^2$.
Right: Ford disks visible in the window $[0,1]^2$.}
\label{fig:greedy}
\end{figure}

\begin{proposition}\label{pro:rarb}
$\per(S_n)\geq \per(F_n(1/2,3-2\sqrt{2})).$
\end{proposition}
\begin{proof}
Recall that the first two disks disks in $S_n$ have radii $1/2$ and $3-2\sqrt{2}$.
Let $I$ be the line segment between the tangency points of $A_1$ and
$A_2$ with the bottom side of $[0,1]^2$. Because of the greedy
strategy, all disks in $S_n$ that intersect segment $I$ are in $F_n(1/2,3-2\sqrt{2})$.
The radius of every disk in $S_n\setminus F_n(1/2,3-2\sqrt{2})$ is at
least at large as any disk in $F_n(1/2,3-2\sqrt{2})\setminus S_n$.
Therefore, there is a one-to-one correspondence between $S_n$ and
$F_n(1/2,3-2\sqrt{2})$ such that each disk in $S_n$ corresponds to a
disk of the same or smaller radius in $F_n(1/2,3-2\sqrt{2})$.
\end{proof}

Given two tangent disks of radii $r_1$ and $r_2$ that are also tangent
to the $x$-axis, there is a unique disk tangent to both these disks
and the $x$-axis, and its radius $r_3$ satisfies $r_3^{-1/2} =
r_1^{-1/2} + r_2^{-1/2}$. Observe that $r_3=r_3(r_1,r_2)$ is a continuous and
monotonically increasing function of both variables, $r_1$ and $r_2$.
Therefore, if $r_1 \leq r'_1$ and $r_2 \leq r'_2$, then
\begin{equation}\label{eq:rarb}
\per(F_n(r_1,r_2)) \leq \per(F_n(r'_1,r'_2)).
\end{equation}
This observation allows us to bound $\per(S_n)$ from below by the perimeter
of a finite subfamily of Ford disks.

The \emph{Ford disks}~\cite{For38} are a packing of an \emph{infinite}
set of disks in the halfplane $\{(x,y)\in \mathbb{R}^2: y\geq 0\}$,
where each disk is tangent to the $x$-axis from above.
Every pair $(p,q)\in \mathbb{N}^2$ of relative prime positive integers
defines a Ford circle $C_{p,q}$ of radius $1/(2q^2)$ centered at
$(p/q,1/(2q^2))$; see Fig.~\ref{fig:greedy}~(right).
The Ford disks $C_{p,1}$ have the largest radius 1/2; all other Ford disks
have smaller radii and each is tangent to two larger Ford disks.
Hence, the set of the $n$ largest Ford disks that touch the unit
segment $[0,1]$ is exactly $F_n(1/2,1/2)$.

\begin{proposition}\label{pro:Ford}
$\per(F_n(1/2,1/2)) =\Omega(\log n)$.
\end{proposition}
\begin{proof}
For a positive integer $Q$, the number of Ford disks of radius at
least $\frac{1}{2Q^2}$ touching the unit segment $[0,1]$ is
$1+\sum_{q=1}^Q \varphi(q)$, where $\varphi(.)$ is Euler's totient function,
\ie, the number positive integers less than or equal to $q$ that are
relatively prime to $q$. It is known~\cite[Theorem 3.7, p.~62]{Ap76}
that
$$\sum_{q=1}^Q \varphi(q) = \frac{3}{\pi^2} Q^2 +O(Q \log Q). $$
Hence, for a suitably large $Q=\Theta(\sqrt{n})$, there exists exactly
$n$ Ford disks of radius at least $\frac{1}{2Q^2}$ that touch $[0,1]$.
Let $F_n(1/2,1/2)$ be the subset of these $n$ Ford disks.
Then we have
$$ \per(F_n)=\sum_{q=1}^Q \varphi(q) \cdot 2\pi \cdot \frac{1}{2q^2}
=\pi\sum_{q=1}^Q \frac{\varphi(q)}{q^2}.$$

It is also known~\cite[Exercise 6, p.~71]{Ap76} that
$$ \sum_{q=1}^Q \frac{\varphi(q)}{q^2} =
\frac{6}{\pi^2} \ln {Q} + O\left(\frac{\log{Q}}{Q}\right). $$

Using this estimate, we have
$$ \per(F_n) = \pi \left( \sum_{q=1}^Q \frac{\varphi(q)}{q^2}\right)
=\Omega(\log Q) =\Omega(\log \sqrt{n}) =\Omega(\log n), $$
as claimed.
\end{proof}

The bounds in Propositions~\ref{pro:rarb}-\ref{pro:Ford} in
conjunction with~\eqref{eq:rarb} yield
\begin{align*}
\per(S_n) &\geq \per(F_n(1/2,3-2\sqrt{2}))
\geq \per(F_n(3-2\sqrt{2},3-2\sqrt{2})) \\
&= \Omega(\per(F_n(1/2,1/2))) = \Omega(\log n).
\end{align*}

When $C$ is a disk and the container $D$ is any other convex body, the above
argument goes through and shows that a greedy packing $S_n$ has total
perimeter $\per(S)=\Omega(\log n)$, where the constant of proportionality
depends on $D$. However, when $C$ is not a circular
disk, the theory of Apollonian circles does not apply.

\section{Homothets Touching the Boundary: Proof of
  Theorem~\ref{thm:boundary}}\label{sec:boundary}

The upper bound $\per(S)=O(\log n)$ follows from Proposition~\ref{pro:boundary}.
It remains to construct a packing $S$ of perimeter $\per(S)=\Omega(\log n)$
for given $C$ and $D$. Let $C$ and $D$ be two convex bodies with bounded
description complexity. We wish to argue analogously to the case of disks
in a square. Therefore, we choose an arc $\gamma\subset \partial D$ that
is smooth and sufficiently ``flat,'' but contains no side parallel to a
corresponding side of $C$. Then we build a hierarchy of homothets of $C$
touching the arc $\gamma$, so that the depth of the hierarchy is $O(\log n)$,
and the homothety factors decrease by a constant between two consecutive levels.

\begin{figure}[htbp]
\centering
\includegraphics[width=.6\columnwidth]{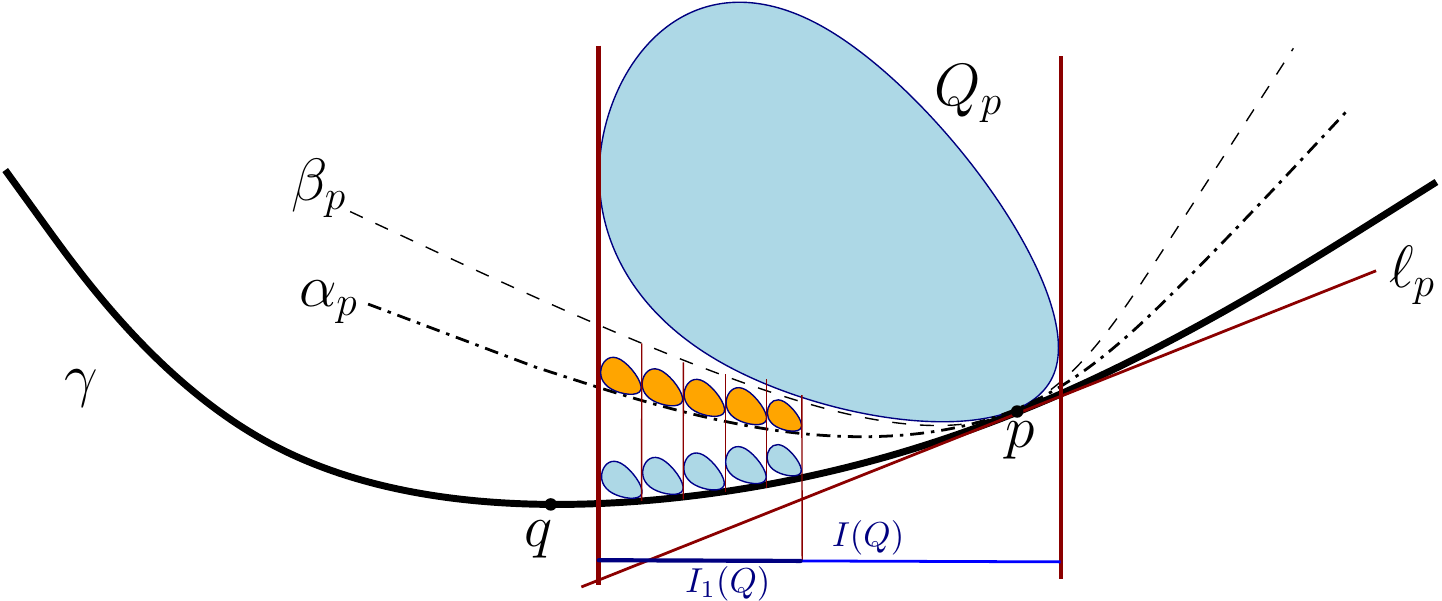}
\caption{If a homothet $C_p$ is tangent to $\gamma\subset \partial D$ at point $p$,
then there are polynomials $\alpha_p$ and $\beta_p$ that separate $\gamma$ from $C_p$.
We can place a constant number of congruent homothets of $C$ between
$\alpha_p$ and $\beta_p$ whose vertical projections cover
$I_1(Q)$. These homothets can be translated vertically down to touch $\gamma$.}
\label{fig:separate}
\end{figure}

We choose an arc $\gamma\subset \partial D$ as follows. If $D$ has a
side with some direction $\mathbf{d}\in \SH$ such that $C$ has no parallel
side of the same direction $\mathbf{d}$, then let $\gamma$ be this side of $D$.
Otherwise, $\partial D$ contains an algebraic curve $\gamma_1$ of degree
2 or higher. Let $q\in \gamma_1$ be an interior point of this curve such
that $\gamma_1$ is twice differentiable at $q$. Assume, after a rigid transformation
of $D$ if necessary, that $q=(0,0)$ is the origin and the supporting line of $D$
at $q$ is the $x$-axis. By the inverse function theorem, there is an arc
$\gamma_2\subseteq \gamma_1$, containing $q$, such that $\gamma_2$ is the graph
of a twice differentiable function of $x$. Finally, let $\gamma\subset \gamma_2$
be an arc such that the part of $\partial C$ that has the same tangent lines as
$\gamma_2$ contains no segments (sides).

For every point $p\in \gamma$, let $p=(x_p,y_p)$, and let $s_p$ be the slope of the
tangent line of $D$ at $p$. Then the tangent line of $D$ at $p\in \gamma$ is
$\ell_p(x)=s_p(x-x_p)$. For any homothet $Q$ of $C$, let $Q_p$ denote
a translate of $Q$ tangent to $\ell_p$ at point $p$ (Fig.~\ref{fig:separate}).
If both $C$ and $D$ have bounded description complexity, then there are
constants $\rho_0>0$, $\kappa, \in \NN$ and $A<B$, such that for every point
$p\in \gamma$ and every homothety factor $\rho$, $0<\rho<\rho_0$, the polynomials
$$
\alpha_p(x)=A|x-x_p|^\kappa+s_p(x-x_p),
\hspace{8mm}\mbox{ \rm and}\hspace{8mm}
\beta_p(x)=B|x-x_p|^\kappa+s_p(x-x_p),
$$
separate $\gamma$ from the convex body $Q_p=(\rho C)_p$.

Similarly to the proof of Theorem~\ref{thm:disks}, the construction is guided
by nested \emph{projection intervals}. Let $Q=(\rho C)_p$ be a homothet of $C$ that
lies in $D$ and is tangent to $\gamma$ at point $p\in \gamma$. Denote by
$I(Q)$ the vertical projection of $Q$ to the $x$-axis. For $k=1,\ldots $,
we recursively define disjoint intervals or interval pairs $I_k(Q)\subset I(Q)$
of length $|I_k(Q)|=|I(Q)|/2^k$. During the recursion, we maintain the
invariant that the set $J_k(Q)=I(Q)\setminus \bigcup_{j<k} I_{j}(Q)$
is an interval of length $|I(Q)|/2^{k-1}$ that contains $x_p$.
Assume that $I_1(Q),\ldots , I_{k-1}(Q)$ have been defined, and we need
to choose $I_k(Q)\subset J_k(Q)$.
If $x_p$ lies in the central one quarter of $J_k(Q)$,
then let $I_k(Q)$ be a pair of intervals that consists of the
left and right \emph{quarters} of $J_k(Q)$. If $x_p$ lies to
the left (right) of the central one quarter of $J_k(Q)$,
then let $I_k(Q)$ be the right (left) \emph{half} of
$J_k(Q)$. It is now an easy matter to check (by induction on $k$)
that $|x-x_p|\geq |I(Q)|/8^k$ for all $x\in I_k(Q)$.
Consequently,
\begin{equation}\label{eq:sep}
\beta_p(x)-\alpha_p(x)\geq (B-A) \cdot \left(\frac{|I(Q)|}{8^k}\right)^\kappa
\end{equation}
for all $x\in I_k(Q)$. There is a constant $\mu>0$ such that a
homothet $\mu^k Q$ with arbitrary projection interval in $I_k(Q)$ fits
between the curves $\alpha_p$ and $\beta_p$. Refer to
Fig.~\ref{fig:separate}. Therefore we can populate the region between
curves $\alpha_p$ and $\beta_p$ and above $I_k(Q)$ with homothets $\rho Q$,
of homothety factors $\mu^k/2<\rho \leq \mu^k$, such that their
projection intervals are pairwise disjoint and cover $I_k(Q)$. By
translating these homothets vertically until they touch $\gamma$,
they remain disjoint from $Q$ and preserve their projection intervals.
We can now repeat the construction of the previous section and
obtain $\lceil \log_{(2/\mu)} n\rceil$ layers of homothets touching $\gamma$,
such that the total length of the projections of the homothets in each layer
is $\Theta(1)$. Consequently, the total perimeter of the homothets in each layer
is $\Theta(1)$, and the overall perimeter of the packing is
$\Theta(\log n)$, as required.
\endproof

\section{Bounds in Term of the Escape Distance: Proof of Theorem~\ref{thm:homothets}}
\label{sec:escape}

\noindent{\bf Upper bound.} Let $S=\{C_1,\ldots,C_n\}$ be a packing of $n$
homothets of a convex body $C$ in a container $D$ such that $D$ is a convex
polygon parallel to $C$. For each element $C_i\in S$, $\esc(C_i)$ is the
distance between a side of $D$ and a corresponding side of $C_i$. For each
side $a$ of $D$, let $S_a\subseteq S$ denote the set of $C_i\in S$
for which $a$ is the closest side of $D$ (ties are broken arbitrarily).
Since $D$ has finitely many sides, it is enough to show that for each
side $a$ of $D$, we have
$$\per(S_a) \leq \rho_a(C,D) \left(\per(D) +\esc(S)\right) \,
\frac{\log |S_a|}{\log \log |S_a|}, $$
where $\rho_a(C,D)$ depends on $a$, $C$ and $D$ only.

Suppose that $S_a=\{C_1,\ldots , C_n\}$ is a packing of $n$ homothets of $C$
such that $\esc(C_i)$ equals the distance between $C_i$ and side $a$ of $D$.
Assume for convenience that $a$ is horizontal.
Let $c\subset \partial C$ be the side of $C$ corresponding to the side $a$ of $D$.
Let $\rho_1=\per(C)/|c|$, and then we can write $\per(C)=\rho_1|c|$. Refer to
Fig.~\ref{fig:parallel}~(left).

Denote by $b\subset c$ the line segment of length $|b|=|c|/2$ with the
same midpoint as $c$. Since $C$ is a convex body,
the two vertical lines though the two endpoints of $b$ intersect
$C$ in two line segments denoted $h_1$ and $h_2$, respectively. Let
$\rho_2=\min(|h_1|,|h_2|)/|b|$, and then $\min(|h_1|,|h_2|) = \rho_2|b|$.
By convexity, every vertical line that intersects segment $b$ intersects $C$
in a vertical segment of length at least $\rho_2|b|$.
Note that $\rho_1$ and $\rho_2$ are constants depending on $C$ and $D$.
For each homothet $C_i\in S_a$, let $b_i\subset \partial C_i$
be the homothetic copy of segment $b\subset \partial C$.

\begin{figure}[htbp]
\centering
\includegraphics[width=.6\columnwidth]{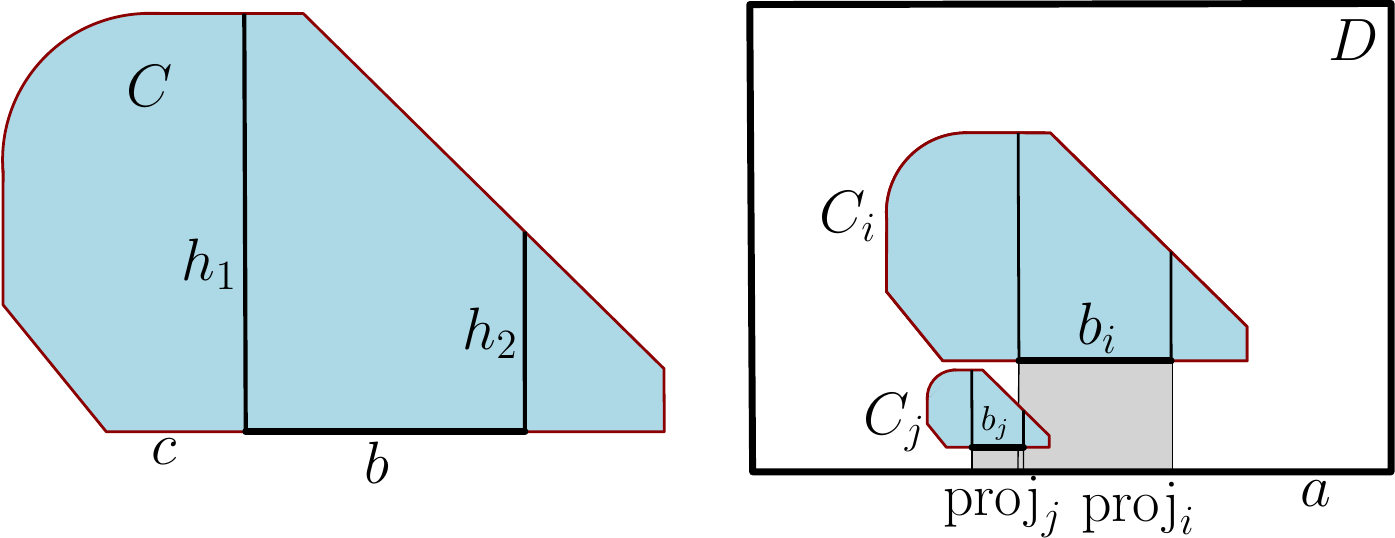}
\caption{Left: A convex body $C$ with a horizontal side $c$. The segment
$b\subset c$ has length $|b|=|c|/2$, and the vertical segments $h_1$ and $h_2$
are incident to the endpoints of $b$.
Right: Two homothets, $C_i$ and $C_j$, in a convex container $D$.
The vertical projections of $b_i$ and $b_j$ onto the horizontal
side $a$ are $\proj_i$ and $\proj_j$.}
\label{fig:parallel}
\end{figure}

Put $\lambda=2\lceil \log n/\log \log n\rceil$.
Partition $S_a$ into two subsets $S_a=S_{\rm far}\cup S_{\rm close}$ as follows.
For each $C_i\in S_a$, let $C_i \in S_{\rm close}$ if
$\esc(C_i)< \rho_2 |b_i|/\lambda$, and $C_i\in S_{\rm far}$ otherwise.
For each homothet $C_i\in S_{\rm close}$, let $\proj_i\subseteq a$
denote the vertical projection of segment $b_i$ onto the
horizontal side $a$ (refer to Fig.~\ref{fig:parallel}, right).
The perimeter of each $C_i\in S_a$ is
$\per(C_i)=\rho_1|c_i|=2\rho_1|b_i| =2\rho_1 |\proj_i|$.
We have
\begin{equation} \label{E4}
\per(S_{\rm far}) =
\sum_{C_i\in S_{\rm far}} \per(C_i) =
\sum_{C_i\in S_{\rm far}} 2\rho_1|b_i| \leq
\sum_{C_i\in S_{\rm far}} 2\rho_1 \frac{\esc(C_i) \, \lambda}{\rho_2} \leq
\frac{2\rho_1 \esc(S)}{\rho_2} \, \lambda.
\end{equation}

It remains the estimate $\per(S_{\rm close})$ as an expression of $\lambda$.
\begin{equation}\label{eq:close}
\sum_{C_i\in S_{\rm close}} \per(C_i)= 2\rho_1 \sum_{C_i\in S_{\rm close}} |\proj_i|.
\end{equation}

Define the \emph{depth} function for every point of the horizontal side $a$ by
$$d:a\rightarrow \NN, \hspace{7mm} d(x)=|\{C_i\in S_{\rm close}: x\in \proj_i\}|.$$
That is, $d(x)$ is the number of homothets such that the vertical projection of
segment $b_i$ contains point $x$. For every positive integer $k \in \NN$, let
$$I_k=\{x\in a: d(x)\geq k\},$$
that is, $I_k$ is the set of points of depth at least $k$. Since
$S_{\rm close}$ is finite, the set $I_k\subseteq a$ is measurable.
Denote by $|I_k|$ the measure (total length) of $I_k$.
By definition, we have $ |a| \geq |I_1| \geq |I_2| \geq \ldots .$
A standard double counting for the integral $\int_{x\in a} d(x) \ {\rm d}x$ yields
\begin{equation}\label{eq:double}
\sum_{C_i\in S_{\rm close}} |\proj_i| = \sum_{k=1}^\infty |I_k|.
\end{equation}

If $d(x)=k$ for some point $x\in a$, then $k$ segments $b_i$,
lie above $x$. Each $C_i\in S_{\rm close}$ is at distance
$\esc(C_i)< \rho_2 |b_i|/\lambda$ from $a$. Suppose that $\proj_i$ and
$\proj_j$ intersect for $C_i,C_j\in S_{\rm close}$ (Fig.~\ref{fig:parallel}, right).
Then one of them has to be closer to $a$ than the other:
we may assume w.l.o.g. $\esc(C_j)<\esc(C_i)$. Now a vertical segment
between $b_i \subset C_i$ and $\proj_i \subset a$ intersects $b_j$.
The length of this intersection segment satisfies
$\rho_2|b_j| \leq \esc(C_i)< \rho_2 |b_i|/\lambda$.
Consequently, $|b_j| < |b_i|/\lambda$ (or, equivalently,
$|\proj_j|<|\proj_i|/\lambda$) holds for any consecutive homothets
above point $x \in a$. In particular, for the $k$-th smallest
projection containing $x \in a$, we have
$|\proj_k|\leq |a|/\lambda^{k-1} = |a| \lambda^{1-k}$.

We claim that
\begin{equation}\label{eq:claim}
|I_k|\leq |a|\lambda^{\lambda-k} \hspace{7mm} \mbox{\rm for }  k\geq \lambda+1.
\end{equation}
Suppose, to the contrary,
that $|I_k|>|a|\lambda^{\lambda-k}$ for some $k\geq \lambda+1$.
Then there are homothets $C_i\in S_{\rm close}$ of side lengths at most
$|a|/\lambda^{k-1}$, that jointly project into $I_k$. Assuming that
$|I_k|>|a|\lambda^{\lambda-k}$, it follows that the number of these
homothets is at least
$$\frac{|a|\lambda^{\lambda-k}}{|a|\lambda^{1-k}}
= \lambda^{\lambda-1}=
\left(2\left\lceil \frac{\log n}{\log \log
  n}\right\rceil\right)^{2\lceil \frac{\log n}{\log \log n}  \rceil
  -1}>n, $$
contradicting the fact that $S_{\rm close}\subseteq S$ has at most $n$ elements.
Combining \eqref{eq:close}, \eqref{eq:double}, and \eqref{eq:claim}, we conclude that
\begin{align} \label{E5}
\per(S_{\rm close}) &=
2 \rho_1\sum_{k=1}^\infty |I_k|\leq
2 \rho_1 \left(\lambda |I_1|+\sum_{k=\lambda+1}^\infty |I_k| \right) \leq
2 \rho_1 \left(\lambda + \sum_{j=1}^\infty \frac{1}{\lambda^j}\right) |a| \nonumber \\
&\leq 2 \rho_1 (\lambda +1) \, \per(D).
\end{align}

Putting~\eqref{E4} and~\eqref{E5} together yields
\begin{align*}
\per(S_a) &= \per(S_{\rm close})+\per(S_{\rm far}) \leq
2 \rho_1 \left((\lambda +1) \, \per(D) + \frac{\esc(S)}{\rho_2} \, \lambda\right)\\
&\leq \rho(C,D) \left(\per(D) +\esc(S)\right) \, \lambda
= \rho(C,D) \left(\per(D) +\esc(S)\right) \, \frac{\log n}{\log \log n},
\end{align*}
for a suitable $\rho(C,D)$ depending on $C$ and $D$, as required;
here we set $\rho(C,D)=2 \rho_1 \max(2,1/\rho_2)$.

\paragraph{Lower bound for squares.}
We first confirm the given lower bound for squares, \ie,
we construct a packing $S$ of $O(n)$ axis-aligned squares
in the unit square $U=[0,1]^2$ with total perimeter
$\Omega((\per(U) +\esc(S)) \log n/\log \log n)$.

Let $n\geq 4$, and put $\lambda= \lfloor \log n/\log \log n\rfloor/2$.
We arrange each square $C_i\in S$ such that
$\per(C_i)=\lambda \, \esc(C_i)$.
We construct $S$ as the union of $\lambda$ subsets
$S=\bigcup_{j=1}^\lambda S_j$, where $S_j$ is a set
of congruent squares, at the same distance from the bottom side of
$U$.

Let $S_1$ be a singleton set consisting of one square
of side length $1/4$ (and perimeter 1) at distance $1/\lambda$
from the bottom side of $U$.
Let $S_2$ be a set of $2\lambda$ squares of side length $1/(4\cdot
2\lambda)$ (and perimeter $1/(2\lambda)$), each at distance
$1/(2\lambda^2)$ from the bottom side of $U$. Note that these squares
lie strictly below the first square in $S_1$, since
$1/(8\lambda)+1/(2\lambda^2)<1/\lambda$. The total length of the vertical
projections of the squares in $S_2$ is $2\lambda\cdot 1/(8\lambda)=1/4$.

Similarly, for $j=3\ldots ,\lambda$, let $S_j$ be a set of
$(2\lambda)^{j-1}$ squares of side length $\frac{1}{4\cdot
  (2\lambda)^{j-1}}$ (and perimeter $1/(2\lambda)^{j-1}$), each at
distance $1/(2^{j-1}\lambda^j)$ from the bottom side of $U$. These
squares lie strictly below any square in $S_{j-1}$; and the total
length of their vertical projections onto the $x$-axis is
$(2\lambda)^{j-1}\cdot \frac{1}{4\cdot (2\lambda)^{j-1}}=1/4$.

The number of squares in $S=\bigcup_{j=1}^\lambda S_j$ is
$$\sum_{j=1}^\lambda (2\lambda)^{j-1} = \Theta\left((2\lambda)^\lambda\right) = O(n).$$

The total distance from the squares to the boundary of $U$ is
$$ \esc(S)= \sum_{j=1}^\lambda (2\lambda)^{j-1}
\frac{1}{2^{j-1}\lambda^j}= \lambda \, \frac{1} {\lambda}=1. $$

The total perimeter of all squares in $S$ is
$$ 4\cdot \sum_{j=1}^\lambda \frac{1}{4} = \lambda=
\Omega\left(\frac{\log n}{\log\log n}\right)
=\Omega\left((\per(U)+\esc(S)) \, \frac{\log n}{\log\log n}\right), $$
as required.

\paragraph{General lower bound.}
We now establish the lower bound in the general setting.
Given a convex body $C$ and a convex polygon $D$ parallel to $C$,
we construct a packing $S$ of $O(n)$ positive homothets of $C$ in
$D$ with total perimeter $\Omega((\per(D) + \esc(S))\log n/\log \log n)$.

Let $a$ be an arbitrary side of $D$. Assume w.l.o.g. that $a$ is
horizontal. Let $U_C$ be the minimum axis-aligned square containing $C$.
Clearly, we have $\frac{1}{2}\, \per(U_C) \leq \per(C) \leq \per(U_C)$.
We first construct a packing $S_U$ of $O(n)$ axis-aligned squares
in $D$ 
such that for each square $U_i\in S_U$, $\esc(U_i)$ equals the distance
from the horizontal side $a$. We then obtain the packing $S$
by inscribing a homothet $C_i$ of $C$ in each square $U_i\in S_U$
such that $C_i$ touches the bottom side of $U_i$. Consequently, we have
$\per(S) \geq \per(S_U)/2$ and $\esc(S)=\esc(S_U)$, since
$\esc(C_i)=\esc(U_i)$ for each square $U_i\in S_U$.

It remains to construct the square packing $S_U$.
Let $U(a)$ be a maximal axis-aligned square contained in $D$ such that its bottom
side is contained in $a$. $S_U$ is a packing of squares in $U(a)$ that
is homothetic with the packing of squares in the unit square $U$
described previously. Put $\rho_1=\per(U(a))/\per(U)=\per(U(a))/4$.
We have
$ \per(S) \geq \frac{1}{4} \, \rho_1 \, \,
\Omega\left((\per(U)+\esc(S)) \, \frac{\log n}{\log\log n}\right),$
or
$$ \per(S) \geq \rho(C,D) \left((\per(D)+\esc(S)) \,
\frac{\log n}{\log\log n}\right),
$$
where $\rho(C,D)$ is a factor depending on $C$ and $D$, as required.
\endproof


\begin{thebibliography}{99}

\bibitem{Ap76}
T. M. Apostol,
\emph{Introduction to Analytic Number Theory},
Springer, New York, 1976.

\bibitem{Ar98}
S. Arora, Polynomial time approximation schemes for
Euclidean traveling salesman and other geometric problems,
\emph{J. ACM} \textbf{45(5)} (1998), 753--782.

\bibitem{Bez13}
K.~Bezdek, On a strong version of the Kepler conjecture,
\emph{Mathematika} {\bf 59(1)} (2013), 23--30.

\bibitem{BE97} M. Bern and D. Eppstein,
Approximation algorithms for geometric problems,
in \emph{Approximation Algorithms for NP-hard Problems},
(D. S. Hochbaum, ed.), PWS Publishing Company, Boston, MA, 1997,
pp.~296--345.

\bibitem{BMP05}
P.~Brass, W.~O.~J.~Moser, and J.~Pach,
\emph{Research Problems in Discrete Geometry},
Springer, New York, 2005.

\bibitem {BGK+05} M. de Berg, J. Gudmundsson, M. J. Katz,
C. Levcopoulos, M. H. Overmars, and A. F. van der Stappen,
TSP with neighborhoods of varying size,
\emph{J.~Algorithms} \textbf{57(1)} (2005), 22--36.

\bibitem{DM03}
A.~Dumitrescu and J.~S.~B.~Mitchell,
Approximation algorithms for TSP with neighborhoods in the plane,
\emph{J.~Algorithms} {\bf 48(1)} (2003), 135--159.

\bibitem{DT11}
A.~Dumitrescu and C.~D. T\'oth,
Minimum weight convex Steiner partitions,
\emph{Algorithmica} {\bf 60(3)} (2011), 627--652.

\bibitem{DT13}
A.~Dumitrescu and C.~D. T\'oth,
The traveling salesman problem for lines, balls and planes,
in \emph{Proc. 24th ACM-SIAM Symposium on Discrete Algorithms}, 2013,
SIAM, pp.~828--843.



\bibitem{For38}
K.~R. Ford, Fractions,
\emph{The American Mathematical Monthly} \textbf{45(9)} (1938), 586--601.

\bibitem{GM14}
A.~Glazyrin and F.~Mori\'c,
Upper bounds for the perimeter of plane convex bodies,
\emph{Acta Mathematica Hungarica} (2013), to appear.

\bibitem{GLM+05}
R.~L. Graham, J.~C. Lagarias, C.~L. Mallows, A.~R. Wilks, and C.~H. Yan,
Apollonian circle packings: geometry and group theory I: The Apollonian group,
\emph{Discrete Comput. Geom.} \textbf{34} (2005), 547--585.

\bibitem{Hal13}
T.~C.~Hales,
The strong dodecahedral conjecture and {F}ejes {T}\'oth's conjecture
on sphere packings with kissing number twelve, in \emph{Discrete
  Geometry and Optimization}, vol.~69 of Fields Communications, 2013,
Springer, 2013, pp.~121--132.


\bibitem{LL84}
C.~Levcopoulos and A.~Lingas,
Bounds on the length of convex partitions of polygons,
in \emph{Proc. 4th Annual Conference on
Foundations of Software Technology and Theoretical Computer Science},
LNCS~181, 1984, Springer, pp.~279--295.


\bibitem{MM95}
C.~Mata and J.~S.~B. Mitchell,
Approximation algorithms for geometric tour and network design problems,
in \emph{Proc. 11th ACM Symposium on Computational Geometry},
1995, ACM, pp.~360--369.

\bibitem{Mi10}
J.~S.~B. Mitchell,
A constant-factor approximation algorithm for TSP with
pairwise-disjoint connected neighborhoods in the plane,
in \emph{Proc. 26th ACM Symposium on Computational Geometry},
2010, ACM, pp.~183--191.






\end{thebibliography}
\end{document}